\documentclass[11pt]{amsart}
\usepackage{graphicx}

\usepackage{amssymb}
\usepackage{epstopdf}
\usepackage{amsmath}
\usepackage{amscd}
\usepackage{amsthm,amsfonts,amsxtra,amssymb,amscd}
\usepackage[all]{xy}
\usepackage{enumerate}
\usepackage{hyperref}

\theoremstyle{plain}
\newtheorem*{lemma*}{Lemma}

\newtheorem*{theorem*}{Theorem}
\newtheorem{theorem}[subsection]{Theorem}
\newtheorem*{proposition*}{Proposition}
\newtheorem{proposition}[subsection]{Proposition}
\newtheorem*{corollary*}{Corollary}
\newtheorem{corollary}[subsection]{Corollary}
\theoremstyle{definition}
\newtheorem*{definition*}{Definition}
\newtheorem{definition}[subsection]{Definition}
\newtheorem*{example*}{Example}
\newtheorem{example}[subsection]{Example}
\theoremstyle{remark}
\newtheorem*{remark*}{Remark}
\newtheorem{remark}[subsection]{Remark}

 \newcommand{\C}{{\mathbb C}}
\newcommand{\R}{{\mathbb R}}
\newcommand{\Z}{{\mathbb Z}}
\renewcommand{\ll}{{\langle}}
\newcommand{\rr}{{\rangle}}
\newcommand{\reg}{{\mathrm{reg}}}
\newcommand{\CR}{{\mathcal{R}}}

\newcommand{\s}{{\bf{s}}}
\begin{document}

\title{ A remark on  the  convolution with the Box Spline}

 \author{ Mich{\`e}le Vergne}
\date{}
\maketitle

\begin{abstract}
The semi-discrete convolution with the Box Spline is an important tool in approximation theory.
We give a formula for the difference between semi-discrete convolution and convolution with the Box Spline. This formula involves multiple Bernoulli polynomials .
\end{abstract}


\section{Box Splines and semi-discrete convolution}
    Let $V$ be a $n$-dimensional real vector space equipped with a lattice $\Lambda$. If we choose a basis of the lattice $\Lambda$, then we may identify $V$ with $\R^n$ and $\Lambda$ with $\Z^n$.   We choose here the Lebesgue measure $dv$ associated to the lattice $\Lambda$.

    Let $X=[a_1,a_2,\ldots, a_N]$ be   a sequence (a multiset)  of  $N$ non zero vectors in $\Lambda$.

    The {\bf zonotope} $Z(X)$
associated with $X$ is the polytope
  $$
  Z(X):=\{\sum_{i=1}^N  t_i a_i\,;\, t_i\in [0,1]\}.
  $$
In other words, $Z(X)$ is the Minkowski sum of the segments $[0,a_i]$
over all  vectors $a_i\in X$.

We denote by $\C[V]$ the space of (complex valued) polynomial functions on $V$.

    Recall that the Box Spline $B(X)$ is the distribution on $V$ such that, for a test function
    $test$ on $V$, we have the equality

    \begin{equation}\label{eq:box}
 \ll B(X), test\rr =\int_{t_1=0}^1\cdots \int_{t_N=0}^1 test(\sum_{i=1}^N t_i a_i) dt_1\cdots dt_N.
\end{equation}

We also note  $\ll B(X), test\rr=\int_V B(X)(v) test(v)$.

The distribution  $B(X)$ is a probability measure supported  on the zonotope $Z(X)$. If $X$ is empty, then $B(X)$ is the $\delta$ distribution on $V$.  For the basic properties of the Box Spline, we refer to \cite{deboor}  (or  \cite{dp1}, chapter 16) .

If $D$ is any distribution on $V$, the convolution $B(X)* D$ is well defined and is again a distribution on $V$.
If $D=f(v)dv$ is a smooth density, then $B(X)*D=F(v)dv$ is a smooth density with $$F(v)=\int_{t_1=0}^1\cdots \int_{t_N=0}^1 f(v-\sum_{i=1}^N t_i a_i) dt_1\cdots dt_N. $$

If $X$ generates $V$, the zonotope is a full dimensional polytope, and $B(X)$  is given by integration against a locally $L^1$-function. Let us describe more precisely where this function is smooth.

We continue to assume that $X$ generates $V$.
An hyperplane of $V$ generated by a subsequence of elements of $X$ is
called  admissible.
An element of $V$ is called (affine) regular, if no translate $v+\lambda$ of $v$ by  any $\lambda$  in the lattice $\Lambda$ lies in an admissible hyperplane. We denote by $V_{\rm reg,aff}$ the open subset of  $V$ consisting of affine regular elements: the set  $V_{\rm reg,aff}$  is the complement of the union of all the translates by $\Lambda$ of admissible hyperplanes.
A connected component $\tau$ of the set of regular elements will be called a (affine) tope (see Figure \ref{topea2}).

\begin{figure}
\begin{center}
  \includegraphics[width=50mm]{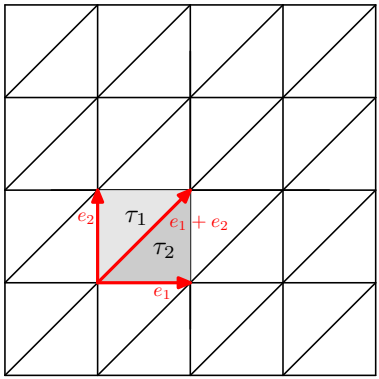}\\
  \caption{Affine topes for $X=[e_1,e_2,e_1+e_2]$}\label{topea2}
\end{center}
\end{figure}

The choice of the Lebesgue measure $dv$ on $V$ allows us to identify distributions and generalized functions: if $F$ is a generalized function, $Fdv$ is a distribution. If the distribution $Fdv$ is given by
  $\ll F dv,test\rr=\int_V f(v) test(v) dv$, with $f(v)$ locally $L^1$, we say that $F$ is locally $L^1$, and we use the same notation for $F$ and the locally $L^1$ function $f$.

A generalized function $b$  on $V$ will be called piecewise polynomial (relative to $X,\Lambda$) if:

$\bullet$ the function $b$ is locally $L^1$,

  $\bullet$ on each tope $\tau$, there exists a polynomial function $b(\tau)$ on $V$ such that the restriction of $b$ to $\tau$ coincides with the restriction of  the polynomial $b(\tau)$ to $\tau$.

If $F$ is a piecewise polynomial function,  we will say that the distribution $Fdv$ is piecewise  polynomial.

      If $X$ generates $V$, the Box Spline  $B(X)$  is a piecewise  polynomial (relative to $(X,\Lambda)$) distribution supported on the zonotope $Z(X)$.

    Let $f$ be a  smooth  function on $V$.
     Then there is  two distributions naturally associated to $X,\Lambda,f$:

      $\bullet$ the piecewise  polynomial distribution $B(X)*_d f$:

     on a test function $test$,
    $$\ll B(X)*_d f,test \rr=\sum_{\lambda\in \Lambda}f(\lambda) \int_{t_1=0}^1\cdots \int_{t_N=0}^1 test(\lambda+\sum_{i=1}^N t_i a_i) dt_1\cdots dt_N.$$

    $\bullet$ the smooth density $B(X)*_c f$:

     on a test function $test$,
    $$\ll B(X)*_c f ,test\rr=\int_{V}f(v) \int_{t_1=0}^1\cdots \int_{t_N=0}^1 test(v+\sum_{i=1}^N t_i a_i) dt_1\cdots dt_N dv.$$

The notations $*_d$ and $*_c$ means discrete, versus continuous.
$B(X)*_d f$ is the convolution of $B(X)$ with the discrete measure $\sum_\lambda f(\lambda)\delta_\lambda$, while
           $B(X)*_c f$  is the usual convolution of $B(X)$ with the smooth density $ f(v)dv$. The subscript $*_c$ is just for emphasis. The operation $*_d$ is denoted $*'$ in \cite{deboor}, \cite{dp1} and is called semi-discrete convolution.

Our aim is to write an explicit formula for the difference $B(X)*_df-B(X)*_cf$.

\bigskip

 We also associate to $a\in X$ three  operators:

 $\bullet$ the partial differential operator
 $$(\partial_{a}f)(v)=\frac{d}{d\epsilon} f(v+\epsilon a),$$

$\bullet$ the difference operator
$$(\nabla_a f)(v)=f(v)-f(v-a),$$

$\bullet$ the integral  operator

$$(I_a f)(v)=\int_0^1 f(v-t a) dt.$$

The operator $I_a$ is the convolution $B([a])*_cf$  with the Box Spline associated to  the sequence with a single element  $a$.

    These three operators respects the space of polynomial functions $\C[V]$ on $V$.
    The Taylor series formula implies that, on the space $\C[V]$, the operator $I_a$ is the invertible operator given by
    $$I_a=\frac{1-e^{-\partial_a}}{\partial_a}=\sum_{j=0}^{\infty} (-1)^j\frac{1}{(j+1)!} \partial_a^j.$$

In particular,  if  $f\in \C[V]$  is a polynomial
\begin{equation}
B(X)*_c f=\left((\prod_{a\in X} \frac{1-e^{-\partial_a}}{\partial_a}) f\right)dv.
\end{equation}

If $I,J$ are subsequences of $X$, we define the operators  $\partial_I=\prod_{a\in I}\partial_a$ and $\nabla_J=\prod_{b\in J}\nabla_b$. They are defined on distributions.

Recall that $\partial_Y B(X)=\nabla_Y B(X\setminus Y)$, if $Y$ is a subsequence of $X$.

A subsequence $Y$ of $X$ will be called long if the sequence $X\setminus Y$ do not generate the vector space $V$.
A  long subsequence $Y$, minimal along the long subsequences, is also called a cocircuit: then $Y=X\setminus H$ where $H$ is an admissible hyperplane.

 In our formula, when $f$ is a polynomial,  $B(X)*_d f-B(X)*_cf$ is naturally expressed in function of  the derivatives $\partial_Yf$ with respect to long subsequences $Y$.

\section{Piecewise smooth distributions}

Our aim  is to write an explicit formula for the difference of the two distributions $B(X)*_d f$ and $B(X)*_c f$. As the first one is a piecewise  polynomial distribution, the second  a smooth density, we will need to introduce an intermediate space of distributions.
We will use "piecewise  smooth distributions".  Let us give a definition.

We continue to assume that $X$ generates $V$.

\begin{definition}
A generalized function  $b$ on $V$ will be called piecewise   smooth (relative to $X,\Lambda$) if:

$\bullet$ the generalized function $b$ is  locally $L^1$,

  $\bullet$ on each tope $\tau$, there exists a smooth function $b(\tau)$ on the full space $V$ such that the restriction of $b$ to $\tau$ coincides with the restriction of  the smooth function $b(\tau)$ to $\tau$.

  \end{definition}

 In this definition, given a tope $\tau$, the function $b$ restricted to $\tau$ (as well as all its derivatives) extends continuously to the closure of $\tau$. However, these extensions do not always coincide  on intersections of the closures of topes.

If $b$ is piecewise  smooth, we then say that the distribution
 $B:=b(v)dv$ (given
  by integration against the locally $L^1$ function $b$) is piecewise  smooth.

It is clear that if we multiply a piecewise  polynomial distribution $B$ by a smooth function, we obtain a piecewise  smooth distribution.
Remark that the space of piecewise smooth distributions is stable by the operators $\nabla_a$, and by convolution with Box Splines $B(Y)$ ($Y$ any subsequence of $X$). However, it is not stable under operators $\partial_a$. For example, $\partial_X B(X)=\nabla_X B(\emptyset)$ is a linear combination of $\delta$ distributions.

%
%
%$\bullet$ $X:=[1]$
%
%Then $b(t):=1$ if $0<t<1$;
%
%
%$\bullet$ $X=[1,1,1]$
%
%Then $$b(t)=1/2*t^2; 0<t<1$$
%
%$$b(t)=-t^2+3*t-3/2; 1<t<2$$
%
%$$b(t)=(1/2)*(t-3)^2; 2<t<3.$$
%
%
%
%
%
%
%
%
%
%
%$\bullet$ $X=[1,1,-1]$
%
%Then $$b(t)=1/2*(t+1)^2; -1<t<0$$
%
%$$b(t):=-t^2+t+1/2; 0<t<1$$
%
%
%$$b(t)=(1/2)*(t-2)^2; 1<t<2.$$
%
%
%
%
%
%
%$\bullet$ $X:=[1,2,-3]$
%
%
%Then
%$$b(t)= (t+3)^2/2; -3<t<-2$$
%
%$$b(t)=(1/6)*t+5/12; -2<t<-1$$
%
%$$b(t)=-(1/12)*t^2+1/3; -1<t<1$$
%
%$$b(t):=-(1/6)*t+5/12; 1<t<2$$
%
%$$b(t):=(1/12)*(t-3)^2; 2<t<3$$
%
%
%
%
%

\section{Multiple Bernoulli periodic polynomials}

Let $U$ be the dual vector space to $V$ and $\Gamma\subset U$  be the dual lattice to $\Lambda$.

 If $Y$ is a subsequence of $X$, we define
$$U_{\reg}(Y)=\{u\in U \, ; \ll a,u \rr \neq
0,\;\mbox{for all} \;a\in Y\}$$
and
$$\Gamma_{\reg}(Y)= \Gamma\cap U_{\reg}(Y).$$

Consider the  periodic function on $V$ given by the (oscillatory) sum
\begin{equation} {W}(X)(v)=\sum_{\gamma \in \Gamma_{\reg}(X)}  \frac{e^{2i\pi \ll v ,\gamma \rr}}{
\prod_{a \in X} 2i\pi \ll a,\gamma \rr}.
\end{equation}

This is well defined as a generalized function on $V$.
In the sense of generalized functions,  we have

\begin{equation}\label{primitive}
\partial_X W(X)(v)=\sum_{\gamma\in \Gamma_{\reg}(X)} e^{2i\pi\ll \gamma,v\rr}.
\end{equation}

We will use this equation to construct ``primitives" of parts of the Poisson formula.

We will call the series $W(X)$ a  multiple Bernoulli series.
 Multiple Bernoulli series have
been  extensively studied by A. Szenes \cite{sze1}.
 They are natural  generalizations of Bernoulli series:
for $\Lambda=\Z \omega$ and $X_k:=[\omega,\omega,\ldots, \omega]$,
where  $\omega$ is repeated $k$ times with $k>0$, the series
$$W(X_k)(t\omega)=\sum_{n\neq 0} \frac{e^{2i\pi nt}}{
(2i\pi n)^k}$$ is equal to $-\frac{1}{k!} B(k,t-[t])$ where
$B(k,t)$ denotes the $k^{\text{th}}$ Bernoulli polynomial in
variable $t$.
In particular, for $k=1$, we have
$W(X_1)(t\omega)= \frac{1}{2}-t+[t]$ (see
Figure \ref{berno}).

\begin{figure}
\begin{center}
  \includegraphics[width=93mm]{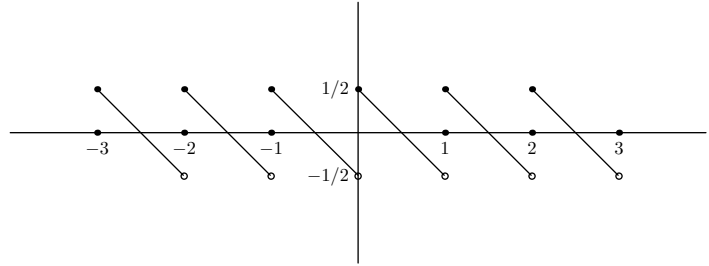}\\
  \caption{Graph of $W(X_1)(t\omega)= \frac{1}{2}-t+[t]$}\label{berno}
\end{center}
\end{figure}

We recall the following proposition \cite{sze1} (see also \cite{brver2}, \cite{BV}).

\begin{proposition}
If $X$ generates $V$, the generalized function $W(X)$ is  piecewise  polynomial (relative  to $(X,\Lambda)$).
\end{proposition}

Thus we will also call $W(X)$ a multiple periodic Bernoulli polynomial.

\bigskip

The above proposition is proved  by reduction to the one variable case.
Indeed, the function $\frac{1}{\prod_{a\in X}\ll a,z\rr}$ can be decomposed in a sum of functions
$\frac{1}{\prod_{i=1}^n\ll a_{j_i},z\rr^{n_i}}$
with respect to a basis  $a_{j_i}$ of $V$ extracted from $X$. This reduces the computation to the one dimensional case.
A. Szenes \cite{sze1} gave an efficient multidimensional explicit residue formula to compute $W(X)$.

\begin{example}\label{bernoulli2}
Let $V=\R e_1\oplus \R e_2$ with lattice $\Lambda=\Z e_1\oplus \Z
e_2$. Let $X=[e_1,e_2,e_1+e_2]$. We write $v\in V$ as $v=v_1
e_1+v_2 e_2$.

We compute the generalized function

 \[\begin{array}{ll}
W(v_1,v_2)&=\displaystyle \sum_{n_1\neq 0, n_2\neq 0, n_1+n_2 \neq
0} \frac{e^{2i\pi(n_1v_1+n_2v_2)}}{(2i\pi n_1)(2i\pi
n_2)(2i\pi(n_1+n_2))}.
\end{array}\]

Then  $W$ is a locally $L^1$-function on $V$, periodic with respect to $\Z e_1+\Z e_2$.
To describe it, it is  sufficient
to write the formulae of $W(v_1,v_2)$ for $0<v_1<1$ and $0<v_2<1$, which we compute (for example using the relation $\frac{1}{n_1 n_2 (n_1+n_2)}=\frac{1}{n_1  (n_1+n_2)^2}+\frac{1}{ n_2 (n_1+n_2)^2}$)
 as :

\[W(v_1,v_2)=\left\{\begin{array}{cl}
   -\frac{1}{6}(1+v_1-2v_2)(v_1-1+v_2)(2v_1-v_2) , & v_1< v_2\\
   -\frac{1}{6}(v_1-2v_2)(v_1-1+v_2)(2v_1-1-v_2),   & v_1> v_2.
       \end{array}\right.\]

       Thus we see that $W$ is a piecewise  polynomial function.

\end{example}

\begin{remark}
If $X$ do not generate $V$, $W(X)$ is not locally $L^1$: take $X=\emptyset$, then, by Poisson formula, $W(\emptyset)$ is the delta distribution of the lattice $\Lambda$.
       \end{remark}

\begin{definition}
A subspace  $\s$ of $V$ generated by a subsequence of elements of $X$ is
called  $X$-admissible.
We denote by $\CR$ the set of  $X$-admissible subspaces of $V$.
We denote by $\CR'$ the set of proper $X$-admissible subspaces.
\end{definition}

The spaces $\s=V$ and $\s=\{0\}$ are among the admissible subspaces of $V$.
The set $\CR'$
consists of all admissible subspaces of $V$, except $\s=V$.

\bigskip

Let $\s$ be an admissible subspace of $V$.
Let us consider  the list $X\setminus \s$, where we have removed from the list $X$ all elements belonging to $\s$.
The projection of the list $X\setminus \s$ on $V/\s$
will be denoted by $X/\s$. The image of the lattice $\Lambda$ in $V/\s$ is a lattice in $V/\s$.
If $X$ generates $V$, $X/\s$ generates $V/\s$.
 Using the projection $V\to V/\s$, we identify    the piecewise  polynomial function $W(X/\s)$ on $V/\s$  to  a piecewise  polynomial function on $V$ constant along the affine spaces $v+\s$.

Define $U_{\rm reg}(X/\s)=U_{\reg}(X\setminus \s)\cap \s^{\perp}$.
Thus $\Gamma_{\rm reg}(X/\s):=\Gamma\cap U_{\rm reg}(X/\s)$ is    the set of elements $\gamma\in \Gamma$  such that:
$$\ll \gamma,s\rr=0 \,\,\,\,{\rm  for\, all\,}\, s\in \s; \hspace{0.5cm}
\ll \gamma,a\rr\neq 0 \,\, \,{\rm  for\, all}\, a\in X\setminus \s.$$

Identifying the dual space to $V/\s$ to the space $\s^{\perp}$, we see that
the function $W(X/\s)$ is  the  function on $V$ given by the series (convergent in the sense of generalized functions)
$$W(X/\s)(v):=\sum_{\gamma\in \Gamma_{\rm reg}(X/\s)} \frac{e^{2i\pi\ll v,\gamma \rr}}{\prod_{a\in X\setminus \s} 2i\pi \ll a,\gamma\rr}.$$
This function is periodic with respect to the lattice $\Lambda$, piecewise polynomial on $V$ (relative to $X$,$\Lambda$) and constant along $v+\s$.

If $\s=V$, the function  $W(X/\s)$  is identically equal to $1$, while if $\s=\{0\}$, we obtain back our series $W(X)$.

\section{A Formula}

Let us now state our formula.
We assume, as before,  that $X$ generates $V$.

For each  $\s\in \CR$, we consider all possible decompositions of
the list $X\setminus \s$ in disjoint lists $I\sqcup J$.
If $f$ is a smooth function, the function
$$F(v)=W(X/\s)(v)(\partial_I \nabla_Jf)(v)$$ is a piecewise  smooth  function on $V$.
If $Y$ is a subsequence of $X$, the convolution
 $B(Y)* F dv$ is well defined and the result is a
 piecewise smooth distribution on $V$ that we denote by
 $B(Y)*_c (W(X/\s)\partial_I \nabla_Jf).$

\begin{theorem}\label{theo:joli}
Let $f$ be a smooth function on $V$.
We have  $$B(X)*_d f -B(X)*_cf=\sum_{\s\in \mathcal R'}\sum_{I\subset X\setminus \s} (-1)^{|I|}
   B((X\cap \s)\sqcup I)*_c(W(X/\s)\partial_I \nabla_Jf).$$
In this formula $J$ is the complement of the sequence  $I$ in $X\setminus \s$.

This equality holds in the space of piecewise (relative to $(X,\Lambda)$)  smooth distributions on $V$, relative to $(X,\Lambda)$.

\end{theorem}

\begin{remark}
If $f$ is a polynomial, the term $B(X)*_cf$ is a polynomial density and all terms of the difference formula are locally polynomial distributions on $V$.
\end{remark}

Before proceeding,
let us comment on the proof. As in \cite{DM} (see also \cite{deboor}), we use the  Poisson formula to compute $B(X)*_df$. Then we group the  terms in the dual lattice $\Gamma$ in strata according to the hyperplane arrangement $\cup_{a\in X}\{a=0\}$. We then use the Bernoulli series as primitives of the corresponding sums.  This way, we introduce the needed derivatives of the function $f$.

\begin{proof}

Let $\CR$ be the set of admissible subspaces of $V$.
We have the disjoint decomposition:
\begin{equation}\label{decompU}
U=\bigsqcup_{\s\in \mathcal R} U_{\reg}(X/\s).
\end{equation}

Let $test$ be a test function on $V$. We compute
$$S:=\int_V (B(X)*_d f)(v)  test(v)=\sum_{\lambda\in \Lambda} f(\lambda) \int_V B(X)(v) test(\lambda+v).$$

We apply Poisson formula to the compactly supported smooth function $$q(w)=f(w)\int_V B(X)(v) test(w+v)$$
as our sum $S$ is equal to $\sum_{\lambda\in \Lambda}q(\lambda).$ We obtain $$S=\sum_{\gamma\in \Gamma} \int_V e^{2i\pi \ll w,\gamma\rr} q(w) dw.$$

The lattice $\Gamma$ is a disjoint union of the sets
 $\Gamma_{\reg}(X/\s)=\Gamma\cap U_{\reg}(X/\s)$ associated to  the admissible subspaces $\s$.
Remark that the set associated to $\s=V$ is $\{\gamma=0\}$. The term in $S$ corresponding to $\gamma=0$ is $\int_V q(w) dw$, that is  $\ll B(X)*_c f,test\rr $.

As in the generalized function sense
$$\sum_{\gamma\in  \Gamma_{\reg}(X/\s)}e^{2i\pi\ll w,\gamma\rr}=\partial_{X\setminus \s} W(X/\s)(w),$$
we obtain
$$S=\sum_{\s\in \mathcal R} \int_V W(X/\s)(w) (-1)^{|X\setminus \s|}\partial_{X\setminus \s}q(w) dw.$$
The function $q(w)$ is product of the two  smooth functions $f(w)$ and
$\int_V B(X)(v) test(w+v).$
 By Leibniz rule,
 $$S:=\sum_{\s}(-1)^{|X\setminus \s|}\sum_{I\sqcup J=X\setminus \s}\int_V \int_V
  W(X/\s)(w)\partial_If(w)  B(X)(v) \partial_Jtest(w+v)  dw.$$

 We first integrate in $v$ and use the equation satisfied by the Box Spline  $$\ll B(X),\partial_b h\rr =-\ll B(X\setminus \{b\}),\nabla_{-b} h\rr.$$
Thus we obtain
$$S= \sum_{\s}\sum_{ I\sqcup J =X\setminus \s}(-1)^{|I|}\int_V \int_V W(X/\s)(w) \partial_I f(w)   B(X\setminus J)(v)(\nabla_{-J} test)(w+v) dw .$$

  Let us integrate in $w$.  We use the invariance of the integral by $\nabla_b$ : $\int_V (\nabla_b f_1)(w) f_2(w) dw =\int_V f_1(w)(\nabla_{-b} f_2)(w) dw.$
  As $b\in J$ is  in $\Lambda$, and $W(X/\s)(w)$ is periodic, $$S=\sum_{\s\in \mathcal R}\sum_{ I\sqcup J =X\setminus \s}(-1)^{|I|}\int_V \int_V
 B(X\setminus J)(v) W(X/\s)(w) \nabla_J \partial_I f(w)  test(w+v) dw .$$

 Writing  $\mathcal R=\{V\}\sqcup \mathcal R'$, we obtain
     the formula of the theorem.

   \end{proof}

\bigskip

 On the space of polynomials, one has $$\nabla_J\partial_I f =(\prod_{b\in J} \frac{1-e^{-\partial_b}}{\partial_b})  \partial_{X\setminus \s}f$$ if $I\sqcup J=X\setminus \s$.

Recall that the space  $D(X)$  of Dahmen-Micchelli polynomials is the space of polynomials on $V$ such that $\partial_Y f=0$ for all long subsequences $Y$. In particular, if $\s$ is a proper subspace, the sequence  $X\setminus \s$ is a long subsequence. So if  $I,J$ are such that $I\sqcup J=X\setminus \s$ and $f\in D(X)$, then $\nabla_J\partial_I f=0$.

As a corollary of our formula, if $p\in D(X)$, we see that $B(X)*_dp=B(X)*_cp$.  Let us state more precisely this result of Dahmen-Micchelli  \cite{DM1} (see also  \cite{dp1}, chapter 16).

\begin{corollary}

If $p\in D(X)$, then   
$$P(v):=B(X)*_d p=\sum_\lambda p(\lambda) B(X)(v-\lambda)$$ is a polynomial function  on $V$, equal to $(\prod_{a\in X}\frac{1-e^{-\partial_a}}{\partial_a})p=B(X)*_c p.$

In this formula, we have identified $B(X)$, $B(X)*_d p$, and $B(X)*_cp$ to piecewise polynomial functions. 

\end{corollary}

\section{Vertices of the arrangement and semi-discrete convolutions.}
We now give a twisted version of Theorem \ref{theo:joli}, where we twist $f$ by an exponential function $e^{2i\pi\ll G,v\rr}$.

The set of characters on $\Lambda$ is the torus $T:=U/\Gamma$.
 If $g\in T$, we denote by $g^{\lambda}$ the corresponding character on $\Lambda$. More precisely if $g$ has representative $G\in U$, then by definition $g^{\lambda}=e^{2i\pi \ll G,\lambda\rr}$.  
Define $$X(g):=\{a\in X; g^{a}=1\}.$$

If $g\in T=U/\Gamma$ has representative $G\in U$, we denote by $g+\Gamma$ the set $G+\Gamma$.

For $a\in X$, introduce the operator $$(\nabla(a,g)f)(v)=f(v)-g^{-a}f(v-a).$$
If $Y$ is a subsequence of $X$, define
$$\nabla_Y^g=\prod_{a\in Y} \nabla(a,g).$$

We  introduce a subset $\mathcal R(g)$ of admissible subspaces, depending on $g$.

\begin{definition}

The admissible space $\s$ is in $\mathcal R(g)$ if the space $(g+\Gamma)\cap \s^{\perp}$
is non empty
\end{definition}

Remark that if $G$ is not in $\Gamma$, then
$V$ is not in the set $\CR(g)$.

\begin{remark}
If $\s\in  \mathcal R(g)$, then all elements of $X\cap \s$ are in $X(g)$. Thus $\mathcal R(g)$ is contained in the set of admissible spaces for $X(g)$. However  the converse does not hold: take  $V=\R \omega$ $X=[2\omega]$,   $\Lambda= \Z \omega$, and $G=\frac{1}{2}\omega^*$. Then $X(g)=X$, so that $V$ is an admissible subspace for $X(g)$. However, $V$ is not in $\CR(g)$.
\end{remark}

If $\s\in \mathcal R(g)$, take
 $g_\s\in (g+\Gamma)\cap \s^{\perp}$.
  Then $(g+\Gamma)\cap \s^{\perp}$ is the translate by $g_\s$ of the lattice $\Gamma\cap \s^{\perp}$.

Define  $$\Gamma_{\reg}(X/\s,g)=(g+\Gamma)\cap U_{\reg}(X/\s)^{\perp}.$$
Thus  $\Gamma_{\reg}(X/\s,g)$ consists  of elements $\xi\in g+\Gamma$ such that
$$\ll \xi,s\rr=0 \,\,\,\,{\rm  for\, all\,}\, s\in \s; \hspace{0.5cm}
\ll \xi,a\rr\neq 0 \,\, \,{\rm  for\, all}\, a\in X\setminus \s.$$

The following series
\begin{equation} {W}(X/\s,g)(v)=\sum_{\xi \in \Gamma_{\reg}(X/\s,g)}  \frac{e^{2i\pi \ll v ,\xi \rr}}{
\prod_{a \in X} 2i\pi \ll a,\xi \rr}
\end{equation}
 is well defined as a generalized function on $V$.

The function   ${W}(X/\s,g)(v)$ is not periodic with respect to $\Lambda$. We have instead the covariance formula

\begin{equation}\label{eq:cov}
W(X/\s,g)(v-\lambda)=g^{-\lambda} W(X/\s,g)(v).
\end{equation}

In the sense of generalized functions,  we have
\begin{equation}\label{primitive}
\partial_{X\setminus \s} W(X/\s,g)(v)=\sum_{\xi\in \Gamma_{\reg}(X/\s,g)} e^{2i\pi\ll \xi,v\rr}.
\end{equation}

We recall the following proposition \cite{sze1} (see also \cite{brver2},\cite{BV}).

\begin{proposition}
The generalized function $W(X/\s,g)$ is  a piecewise   polynomial (relative  to $(X,\Lambda)$) function on $V$.
\end{proposition}

It is proven similarly by reduction to one variable.

\begin{example}
Let $V=\R \omega$,  $\Lambda=\Z \omega$ and $X_k:=[\omega,\omega,\ldots, \omega]$,
where  $\omega$ is repeated $k$ times with $k>0$.
Then $\Gamma=\Z \omega^*$. 
 Then if $z$ is not an integer
$$W(X_k,z\omega^*)(t\omega)=\sum_{n\in \Z} \frac{e^{2i\pi (n+z)t}}{
(2i\pi (n+z))^k}.$$

We have, for example, (see \cite{brver2})
$$W(X_1,z\omega^*)(t\omega)=\frac{e^{2i\pi [t] z}}{1-e^{-2i\pi z}}.$$
$$W(X_2,z\omega^*)(t\omega)=e^{2i\pi [t]z}(\frac{t-[t]}{1-e^{-2i\pi z}}+\frac{1}{(1-e^{-2i\pi z})(1-e^{2i\pi z})}).$$

Here $[t]$ is the integral part of $t$. This function $[t]$ is a constant on each interval $]\ell,\ell+1[$, and $W(X_k,z\omega^*)$ is a locally polynomial function of $t$.
\end{example}

\begin{theorem}\label{theo:joli2}
Let $G\in U$, and $g$ its image in $U/\Gamma.$ Let $f(v)=e^{2i\pi  \ll v,G\rr}h(v)$, where $h$ is a smooth function.
Then
$$B(X)*_d f=\sum_{\s\in \mathcal R(g)} \sum_{I\subset X\setminus \s}  (-1)^{|I|}  B((X\cap \s)\sqcup I)*_c (W(X/\s,g) \partial_I \nabla_{J}^g h).$$

In this formula, $J$  is the complement of $I$ in
$X\setminus \s$.

\end{theorem}

\begin{remark}

If $G\in \Gamma$, then $B(X)*_d f=B(X)*_d h$, and the formula of the theorem above coincide with the formula of Theorem \ref{theo:joli} for $h$:  the set $\mathcal R(g)$ coincide with the set $\mathcal R$, and the term  corresponding to $V$ in the formula of  Theorem \ref{theo:joli2}  is $B(X)*_c h$.

\end{remark}
 \begin{proof}
We proceed in the same way than the proof of Theorem \ref{theo:joli}.
Let $test$ be a test function on $V$. We compute
$S:=\int_V (B(X)*_d f)(v)  test(v)$ by Poisson formula.
If  $$q(w)=h(w)\int_V B(X)(v) test(w+v),$$ we obtain
$$S=\sum_{\gamma\in \Gamma} \int_V e^{2i\pi \ll w,\gamma\rr} e^{2i\pi \ll w,G\rr}q(w) dw.$$
Thus
 $$S=\sum_{\xi\in (g+\Gamma)} \int_V e^{2i\pi \ll w,\xi\rr} q(w) dw.$$

The set $g+\Gamma$ is a disjoint union over $\s\in \mathcal R(g)$ of the sets
 $\Gamma_{\reg}(X/\s,g)=(g+\Gamma)\cap U_{\reg}(X/\s)$, so
$$S=\sum_{\s\in \mathcal R(g)} \int_V W(X/\s,g)(w) (-1)^{|X\setminus \s|}\partial_{X\setminus \s}q(w) dw.$$
Then, using Leibniz rule  for $\partial_a$ , and equation for the Box Spline, we obtain  that $S$ is equal to
$$\sum_{\s\in \mathcal R(g)}\sum_{ I\sqcup J =X\setminus \s}(-1)^{|I|}\int_V \int_V W(X\setminus \s,g)(w) \partial_I f(w)  B(X\setminus J)(v) (\nabla_{-J}test)(w+v) dw .$$

Using the covariance formula  (\ref{eq:cov}) for $W(X\setminus\s,g)$,  we see that
 $$\int_V  W(X\setminus \s,g)(w) f_1(w) (\nabla_{-b}f_2)(w) dw =\int_V  W(X\setminus \s,g)(w) (\nabla(b,g)f_1)(w) f_2(w) dw $$
and  we obtain
the formula of the theorem.
   \end{proof}

\bigskip

Let us
 point out a corollary of this formula.

\begin{definition}
  We say that a point $g\in U/\Gamma$ is a toric vertex of the  arrangement $X$, if $X(g)$  generates $V$.
                  We denote by $\mathcal V(X)$ the set of toric vertices of the arrangement $X$.
  \end{definition}

   If $g$ is a vertex,  there is a basis $\sigma$ of $V$ extracted from $X$ such that $g^{a}=1$, for all $a\in \sigma$. We thus see that   the set $\mathcal V(X)$ is finite.

   If $X$ is unimodular, then $\mathcal V(X)$ is reduced to $g=0$.

  \begin{corollary}(Dahmen-Micchelli)

  Let $g\in \mathcal V(X)$ be a toric vertex of the arrangement $X$ and let $p\in D(X(g))$ be  a polynomial in the Dahmen-Micchelli space for $X(g)$.
  Assume that $g\neq 0$.
   Let $f(\lambda)=g^{\lambda}p(\lambda)$.
  Then  $B(X)*_d f=0$
  \end{corollary}

  \begin{proof}
  We apply the formula of Theorem \ref{theo:joli2} with $h=p$. As $g\neq 0$, all terms $\s\in \mathcal R(g)$ are proper subspaces of $V$.
  Let us show that all the terms  in our formula are $0$.
  Indeed let $I\sqcup J=X\setminus \s$.
                      Let $I'=I\cap X(g)$ and $J'=J\cap X(g)$.
                      Then $I'\sqcup J'=X(g)\setminus \s$ is a long subset of $X(g)$. As $\nabla_{I'}^g=\nabla_{I'}$, we see that
                      $\partial_{I'}\nabla_{J'}^p$ is already equal to $0$.

  \end{proof}

I wish to thank Michel Duflo for comments on this text.

\vspace{5cm}

\thanks{{\bf Mich{\`e}le Vergne}, Institut de Math\'ematiques de Jussieu, Th{\'e}orie des
Groupes, Case 7012, 2 Place Jussieu, 75251 Paris Cedex 05, France;
 }

\thanks{email: vergne@math.jussieu.fr}

\end{document}